\documentclass[11pt]{amsart}
\usepackage{amsmath}
\usepackage{amsfonts,amsmath}
\usepackage{paralist}
\usepackage{multirow}
\usepackage{array}
\usepackage{xcolor}
\usepackage[margin=1in]{geometry}
 \numberwithin{equation}{section}

\newtheorem{theorem}{Theorem}[section]

\newtheorem{lemma}[theorem]{Lemma}

\theoremstyle{definition}

\newcommand{\ot}{\Omega_T }
\newcommand{\opt}{\Omega^\prime_T }
\newcommand{\op}{\Omega^\prime}
\newcommand{\po}{\partial\Omega}
\newcommand{\mdiv}{\textup{div}}

\newcommand{\io}{\int_{\Omega}}
\newcommand{\iop}{\int_{\Omega^\prime}}
\newcommand{\ioT}{\int_{\Omega_{T}}}

\newcommand{\tmax}{T_{0}}
\newcommand{\vp}{\varphi}
\newcommand{\ra}{\rightarrow}

\newcommand{\pe}{\Phi_e}
\newcommand{\ps}{\Phi_s}
\newcommand{\pet}{\Phi_e^{(\tau)}}
\newcommand{\pst}{\Phi_s^{(\tau)}}
\newcommand{\ct}{C^{(\tau)}}
\newcommand{\mt}{M^{(\tau)}}
\newcommand{\lt}{L^{(\tau)}}

\newcommand{\kac}{\kappa(C)}
\newcommand{\kact}{\kappa_\tau(C)}
\newcommand{\kactt}{\kappa_\tau(C^{(\tau)})}
\newcommand{\thc}{C^+}

\newcommand{\ee}{\varepsilon_e}
\newcommand{\pt}{\partial_t}

\newcommand{\vep}{\varepsilon}
\newcommand{\ep}{\varepsilon}
\newcommand{\esup}{\textup{ess sup}}
\newcommand{\einf}{\textup{ess inf}}
\begin{document}




\title[a PDE model
for Lithium-ion batteries ] 
{ Life span of solutions to a PDE model
	for Lithium-ion batteries in high space dimensions }
	\author{Xiangsheng Xu}\thanks
{Department of Mathematics and Statistics, Mississippi State
	University, Mississippi State, MS 39762.
	{\it Email}: xxu@math.msstate.edu.}

	\subjclass{Primary: 35A01,35M10, 35J57, 35K20, 35B50.}
	\keywords{Lithium-ion batteries; mathematical model; the De Giorgi iteration scheme; existence and boundedness of a solution.}
\email{xxu@math.msstate.edu}
	
	\bigskip
	
		\begin{abstract}In this paper we study a system of partial differential equations which models lithium-ion batteries. The system describes the conservation of Lithium and conservation of charges in the solid and electrolyte phases, together with the conservation of energy. The mathematical challenge is due to the fact that the reaction terms in the system involve the hyperbolic sine function along with possible degeneracy in one of the high order terms. We obtain a local existence assertion for the initial boundary problem for the system which offers insight into how long a battery can last.
		\end{abstract}
	\maketitle
	


	\section{Introduction}
	A lithium-ion battery or Li-ion battery (abbreviated as LIB) is a type of rechargeable battery. The technology was largely developed by John Goodenough, Stanley Whittingham, Rachid Yazami, and Akira Yoshino during the 1970s–1980s, and then commercialized by a Sony and Asahi Kasei team led by Yoshio Nishi in 1991. 
	 Lithium-ion batteries are commonly used for portable electronics and electric vehicles and are growing in popularity for military, civilian, and aerospace applications.
A mathematical model describing the key factors of the battery operation can be very helpful for the
	design and optimization of 
	battery performance.
	Based on a macro-homogeneous approach developed by Newman (see \cite{N1,N2})
	several mathematical models have been developed for these purposes (see \cite{CE,DNGST,FDN,F,GW,HMA,WGL,RA}) which include the main physics present in charge/discharge
	processes. 

	A typical Lithium-ion battery cell ($\Omega$) has three regions: a porous negative electrode ($\Omega_a$), a porous positive electrode ($\Omega_c$), and an electro-blocking separator ($\Omega_s$). See Fig.\ref{fig1} below.
	\begin{figure}[h]
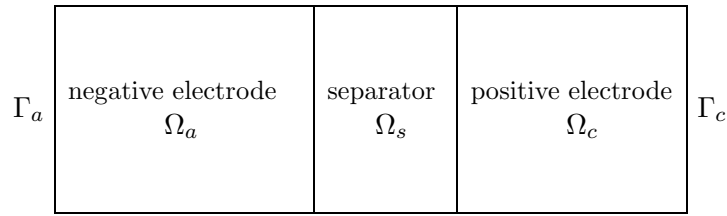

		\centering
	$\Gamma_a$	\begin{tabular}{ | m{8em} | m{4em}| m{7em} | } 
			\hline
			& & \\
				& & \\ 
{\small negative electrode}&  {\small separator}&{\small positive electrode}\\ 
		\hspace{.45in}	
 $\Omega_a$ &   	\hspace{.2in}  $\Omega_s$ & 
		\hspace{.45in} 
 $\Omega_c$  \\ 
			& & \\ 
		 & &  \\ 
			\hline
		\end{tabular}	$\Gamma_c$
	\hspace{-.5in}	\caption{ The domain $\Omega$} 
	\label{fig1}
\end{figure}
In the cell there is stored an electrolyte, which is a concentrated solution containing charged species (lithium ions) that move along the cell in response to an electrochemical potential gradient. In the batteries, lithium ions move from the negative electrode through an electrolyte to the positive electrode during discharge, and back when charging. 
The concentration of the lithium ions in the electrolyte, $C$, the electric potential in the electrodes, $\ps$ , and the electric potential measured by a reference Lithium electrode in the electrolyte, $\pe$, can be modeled, for each $t\in (0,T)$, by the following equations expressing the conservation of various physical quantities \cite{WXZ}: 
\begin{eqnarray}
-\mdiv(\kac\nabla\pe)+\mdiv\left(\alpha_1\kac\nabla\ln C\right)-S_e&=&0\ \ \mbox{in $\Omega$,}\label{pe}\\
-\mdiv(\sigma\nabla\ps)+S_e&=&0\ \ \mbox{in $\Omega^\prime\equiv\Omega_a\cup\Omega_c$,}\label{ps}\\
\ee \pt C-\mdiv(D\nabla C)-\alpha_3S_e&=&0\ \ \mbox{in $\Omega$.}\label{c}
\end{eqnarray}
We further specify the terms in the system. We assume:
\begin{enumerate}
	\item[(H1)] The function $\kac$ lies in the space $C[0,\infty)$, satisfying
	\begin{equation}\label{kacon}
	\mbox{	$\kappa(0)=0,\ \kac>0$ for $C>0$, and $\kappa(C)\geq c_0C^{\alpha_0}$  on $[0,s_0)$ for some $c_0, s_0, \alpha_0>0$}.
	\end{equation}
		\item[(H2)]	$\sigma\in L^\infty(\op) ,\ee\in L^\infty(\Omega), D\in L^\infty(\ot)$ with
		$$\einf_{\op}\sigma>0, \ \ \einf_{\Omega}\ee>0, \ \ \einf_{\ot}D>0.$$
\item[(H3)]	The term $S_e$ represents a function of $C,\ps,\pe$, which has the expression
	\begin{equation}\label{se}
		S_e=S_e(C,\ps,\pe)=\left\{\begin{array}{c}
			\alpha_4\sqrt{C}\sinh\left(\alpha_2(\ps-\pe-U(C))\right)\ \ \mbox{in $\Omega^\prime$,}\\
			0\ \ \mbox{ in $\Omega_s$, }
		\end{array}\right.
	\end{equation}
	where $U$ is a bounded smooth function of $C$.
		\item[(H4)]	The numbers $\alpha_i, i=1,2,3,4$, in \eqref{pe} and \eqref{se} are all positive numbers, whose precise values are determined by the physical properties of the materials involved.
\end{enumerate}

To prescribe the initial boundary conditions, we first set
$$\Gamma_a=\po\cap\partial\Omega_a,\ \ \Gamma_c=\po\cap\partial\Omega_c.$$
We call $\Gamma_a\cup\Gamma_c$ the  external boundary of $\Omega^\prime$ and denote it by $\partial_{\textup{ext}}\Omega^\prime$. Subsequently, we impose
\begin{eqnarray}
\frac{\partial\pe}{\partial \mathbf{n}}&=&0\ \ \mbox{on $\partial\Omega$,}\nonumber\\
\frac{\partial\ps}{\partial \mathbf{n}}&=&0\ \ \mbox{on $\partial\Omega^\prime\setminus\partial_{\textup{ext}}\Omega^\prime$,}\label{bcs}\\
-\sigma\frac{\partial\ps}{\partial \mathbf{n}}&=&I\ \ \mbox{on $\partial_{\textup{ext}}\Omega^\prime$,}\nonumber\\
\frac{\partial C}{\partial \mathbf{n}}&=&0\ \ \mbox{on $\partial\Omega$,}\nonumber\\
C(x,0)&=& C_0(x)\ \ \mbox{on $\Omega$,}\nonumber
\end{eqnarray}
where $\mathbf{n}$ is the unit outward normal to the boundary and $I$ is the current density (a given quantity). 
As in \cite{WXZ}, we assume
\begin{equation}\label{bcs2}
\int_{\partial_{\textup{ext}}\Omega^\prime}I dS=0.
\end{equation}
This condition means that
 no charges are generated or consumed within the battery. Equipped with this, we can simplify the problem a little bit. To this end, we observe that
 \eqref{bcs2} combing with 
\eqref{bcs} yields
\begin{equation}
\int_{\partial\Omega^\prime}\sigma\frac{\partial\ps}{\partial \mathbf{n}}dS=0.\nonumber
\end{equation}
Consequently, (under suitable assumptions on $I$) we can find a function $\phi\in W^{1,2}(\Omega^\prime)\cap L^{\infty}(\Omega^\prime)$ such that
\begin{eqnarray}
-\mdiv(\sigma\nabla\phi)&=&0\ \ \mbox{in $\Omega^\prime$,}\nonumber\\
\frac{\partial\phi}{\partial \mathbf{n}}&=&\frac{\partial\ps}{\partial \mathbf{n}}\ \ \mbox{on $\partial\Omega^\prime$.}\nonumber
\end{eqnarray}
We make the following change of dependent variables
\begin{equation}
\ps\leftarrow \ps-\phi,\ \ \pe\leftarrow\pe-\alpha_1\ln C.\nonumber
\end{equation}
Accordingly, we represent $S_e$ in terms of the new dependent variables
\begin{eqnarray}
S_e&=&\alpha_4\sqrt{C}\sinh\left(\alpha_2(\ps-\pe)+\alpha_2\phi-\alpha_1\alpha_2\ln C-\alpha_2U(C)\right)\chi_{\Omega^\prime}\nonumber\\
&=&\frac{1}{2}\alpha_4\sqrt{C}\left(hC^{-d}e^{\alpha_2(\ps-\pe)}-h^{-1}C^{d}e^{-\alpha_2(\ps-\pe)}\right)\chi_{\Omega^\prime},\nonumber
\end{eqnarray}
where 
$$ d=\alpha_1\alpha_2\ \mbox{ and}\ \ h=e^{\alpha_2(\phi-U(C))}.$$
%
In summary, we obtain the following problem
\begin{eqnarray}
-\mdiv(\kac\nabla\pe)&=&S_e\chi_{\Omega^\prime}\ \ \mbox{in $\ot\equiv \Omega\times(0,T)$,}\label{cpe}\\
-\mdiv(\sigma\nabla\ps)&=&-S_e\ \ \mbox{in $\opt\equiv\Omega^\prime\times(0,T)$,}\label{cps}\\
\ee \pt C-\mdiv(D\nabla C)&=&\alpha_3S_e\chi_{\Omega^\prime}\ \ \mbox{in $\ot$,}\label{cc}\\
\frac{\partial\pe}{\partial \mathbf{n}}&=&0\ \ \mbox{on $\partial\Omega\times(0,T)$,}\label{pebc}\\
\frac{\partial\ps}{\partial \mathbf{n}}&=&0\ \ \mbox{on $\partial\Omega^\prime\times(0,T)$,}\label{psbc}\\
\frac{\partial C}{\partial \mathbf{n}}&=&0\ \ \mbox{on $\partial\Omega\times(0,T)$,}\label{cbc}\\
C(x,0)&=& C_0(x)\ \ \mbox{on $\Omega$, }\label{ini}
\end{eqnarray}
where
 $T>0$.
Note that the introduction of the function $f$ in \cite{WXZ} is unnecessary. 

 If we introduce the function
\begin{equation}
	G(y_1, y_2, y_3)=y_1y_2^{-d}e^{\alpha_2y_3}-y_1^{-1}y_2^{d}e^{-\alpha_2y_3} \ \ \mbox{for $(y_1,y_2,y_3)\in (0,\infty)\times(0,\infty)\times\mathbb{R}$,}\nonumber
\end{equation}
then we can write $S_e$ in the form
\begin{equation}
	S_e=\frac{1}{2}\alpha_4\sqrt{C}G(h, C,\ps-\pe)\chi_{\Omega^\prime}.\nonumber
\end{equation} 
For simplicity, we assume $U=1$. Consequently, we can view $h$ as a given  function with the property
\begin{equation}
	\frac{1}{K}\leq h\leq K\ \ \mbox{for some $K\geq 1$}.\label{hb}
\end{equation}
The case where $U$ is an increasing function of $C$ can be handled in an entirely similar manner. 

 Define the function spaces
 \begin{eqnarray}
 	X_T&=&L^{\infty}(\Omega_{T})\times L^{\infty}(\Omega^\prime_{T})\times C(\overline{\Omega_{T}}),\label{xdef}\\
 	 Y_T&= &  L^\infty(0,T; W^{1,2}(\Omega))\times L^\infty(0,T; W^{1,2}(\op))\times L^2(0,T; W^{1,2}(\Omega)).\label{ydef}
 \end{eqnarray}
 Local existence of a weak solution in $X_T\cap Y_T$ to \eqref{cpe}-\eqref{ini} was already established in \cite{WXZ}. The result there asserts that there is positive number $\tmax$ such that  problem \eqref{cpe}-\eqref{ini} has a solution for  $T=\tmax$. This is to be expected because a battery can last for only a limited period of time. However, the local existence theorem in \cite{WXZ} did not offer any quantitative information on the size of $\tmax$. The objective of this paper is to fill this gap. To be precise, we have:
 \begin{theorem}\label{thm1} Let (H1)-(H4) be satisfied.
 	Assume:
 	\begin{enumerate}
 		\item[\textup{(H5)}]$N\geq 2$ and $\po, \partial\op$ are Lipschitz;
 			\item[\textup{(H6)}] $ h\in C(\overline{\ot})$ is such that \eqref{hb} is satisfied;
 		\item[\textup{(H7)}] $C_0(x)\in C^{\alpha_5}(\overline{\Omega})$ for some $\alpha_5\in (0,1)$ with the property
 		\begin{equation}
 			\min_{\overline{\Omega}}C_0(x)>0;\nonumber
 		\end{equation}
 		\item[\textup{(H8)}] $d>\frac{1}{2}$;
 		\item[\textup{(H9)}]  $|\Omega_s|<|\op|$.
 		\end{enumerate}
 	Then there is a positive number $\tmax$ such that for $T=\tmax$ problem  \eqref{cpe}-\eqref{ini} has a weak solution in $X_T\cap Y_T$ with the property 
 	\begin{eqnarray}
 		\min_{\overline{\ot}} C&>&0,\\
 		\io\pe(x,t)dx+\iop\ps(x,t)dx&=&0.\label{dess}
 	\end{eqnarray}
 Furthermore, the number $\tmax$ satisfies the equation \eqref{tzdf} below.
 	 \end{theorem}
 A careful examination of the proof of \eqref{tzdf} below reveals how $\tmax$ ``explicitly'' depends on the given data. At least,  we can gain an estimate for $\tmax$ (see \eqref{f6} below). Our method is more constructive than the one employed in \cite{WXZ}. Obviously, $\tmax$ represents a lower bound for battery life.

Since $\kappa(0)=0$, we must obtain a positive lower bound for $C$ to avoid degeneracy in \eqref{cpe}. This also removes the singularity of $S_e$ at $C=0$. It turns out that condition (H8) is sufficient to do the deed. See Lemma \ref{lclb} below. It is worth noting that (H8) is also the condition which ensures that 
$S_e$ is a decreasing function of $C$. Since $\sigma$ is taken to be $0$ on $\Omega_s$, equation \eqref{cps} is only satisfied on $\opt$. Condition (H9) is assumed to circumvent the difficulty caused by this. In applications, (H9) makes sense because separators $\Omega_s$ are rather thin.
The reaction term $S_e$ is messy at first glance. A closer examination shows that the function $G$ in $S_e$ is the composition of $b(s)\equiv s-\frac{1}{s}$ with the term $hC^{-d} e^{\alpha_2(\ps-\pe)}$. The function $b(s)$ has two singular points at $0$ and $\infty$. This suggests that we need to establish
$$\ps-\pe\in L^\infty(\opt), \ \ C\in  L^\infty(\ot)\ \ \mbox{with $\einf_{\ot} C>0$.}$$
It turns out that the above is indeed sufficient for an existence assertion for our problem.
We observe
\begin{eqnarray}
	\partial_{y_3}G(y_1, y_2, y_3)&=&\alpha_2\left(y_1y_2^{-d}e^{\alpha_2y_3}+y_1^{-1}y_2^{d}e^{-\alpha_2y_3}\right)\nonumber\\
	&\geq &2\alpha_2>0\ \ \mbox{for $(y_1,y_2,y_3)\in (0,\infty)\times(0,\infty)\times\mathbb{R}$, and}\label{gpr}\\
	\partial_{y_2}G(y_1,y_2,y_3)&=&-\frac{d}{y_2}(y_1y_2^{-d}e^{\alpha_2y_3}+y_1^{-1}y_2^{d}e^{-\alpha_2y_3})\nonumber\\
	&\leq& -\frac{2d}{y_2}<0\ \ \mbox{on $(0,\infty)\times(0,\infty)\times\mathbb{R}$}.\nonumber
\end{eqnarray}
These properties of $G$ will be extensively exploited.

The uniqueness of a solution clearly does not hold. If $(\pe,\ps, C)$ is a solution, so is $(\pe+g(t),\ps+g(t), C)$ for any $g(t)$. It would be interesting to know if the uniqueness holds under \eqref{dess}. 

There is a large body of literature devoted to the mathematical study of lithium-ion batteries, and we will not attempt to make a comprehensive review. As far as we know,  all the current work is concerned with the so-called (P2D) models in which the space dimension is essentially one with the lonely exception of \cite{WXZ}. Here we only mention \cite{DGR,DNGST,FDN,K,WGL}, where
one can find numerical computations and analytic study of this and other related models, with
parameters corresponding to actual devices, that help to highlight the structure
of models and show the relevance to the applications.  As observed in \cite{WXZ},  the gap between $N=1$ and $N>1$ is substantial in terms of mathematical analysis.

The space dimension $N$ is assumed to be bigger than $2$ for the convenience of applying the Sobolev embedding theorem. Our argument still works for $N=2$ under some minor adjustments. H\"{o}lder's inequality is often used without acknowledgment. The same goes for the following two inequalities
$$(|a|+|b|)^\beta\leq 2^{\beta-1}\left(|a|^\beta+|b|^\beta\right)\ \ \mbox{if $\beta>1$,}\ \ (|a|+|b|)^\beta\leq |a|^\beta+|b|^\beta\ \ \mbox{if $0<\beta\leq1$.}$$
When no confusion arises, we suppress the dependence of a function on its variables.  As usual, the letter $c$ is used to denote a generic positive constant.
 
This paper is organized as follows. In Section 2 we collect a few known results that have played essential roles in our development. Section 3 is devoted to the study of approximate problems. It is interesting to note that our approximation scheme is designed in such a way that the global existence of a solution holds.  The proof of Theorem \ref{thm1} is presented in Section 4.
\section{Preliminaries}
In this section we collect a few known results that are useful to us.

Our existence theorem is based upon the following fixed point theorem, which is often called the Leray-Schauder fixed point lemma (\cite{GT}, p.280).
\begin{lemma}\label{lsf}
	Let $B$ be a map from a Banach space $\mathcal{B}$ into itself. Assume:
	\begin{enumerate}
		\item[(LS1)] $B$ is continuous;
		\item[(LS2)] the images of bounded sets of $B$ are precompact;
		\item[(LS3)] there exists a constant $c$ such that
		$$\|z\|_{\mathcal{B}}\leq c$$
		for all $z\in\mathcal{B}$ and $\sigma\in[0,1]$ satisfying $z=\sigma B(z)$.
	\end{enumerate}
	Then $B$ has a fixed point.
\end{lemma}
 \begin{lemma}\label{ynb}
	Let $\{y_n\}, n=0,1,2,\cdots$, be a sequence of positive numbers satisfying the recursive inequalities
	\begin{equation}\label{ynb1}
		y_{n+1}\leq cb^ny_n^{1+\alpha}\ \ \mbox{for some $b>1, c, \alpha\in (0,\infty)$.}
	\end{equation}
	If
	\begin{equation*}
		y_0\leq c^{-\frac{1}{\alpha}}b^{-\frac{1}{\alpha^2}},
	\end{equation*}
	then $\lim_{n\rightarrow\infty}y_n=0$.
\end{lemma}
This lemma can be found in (\cite{D}, p.12).

The following lemma plays a key role in the proof of Theorem \ref{thm1}. It is inspired by a result in \cite{X1}.
\begin{lemma}\label{prop2.2}
	Let $a(\tau)$ be a continuous non-negative function defined on $[0, \tmax]$ for some $\tmax>0$. Suppose that there exist three positive numbers $\vep, \delta, m $ such that
	\begin{equation}\label{f2}
		a(\tau)< \ep e^{ ma^{1+\delta}(\tau)}+m\ \ \mbox{for each $\tau \in[0, \tmax]$}.
	\end{equation}
	Then there exists a positive number $s_0$ determined by $m, \delta$ only such that if \eqref{f2} holds for
	\begin{eqnarray}
		\ep=\ep_0\equiv\frac{1}{m(1+\delta) e^{ms_0^{1+\delta}}s_0^\delta}\label{epz}
	\end{eqnarray} then we have
	\begin{equation}\label{f5}
		a(\tau)< s_0 \ \mbox{for each $\tau \in(0, \tmax]$, provided that $a(0)\leq m$}.
	\end{equation} 
\end{lemma}
\begin{proof}Introduce the function $f_\vep(s)=\ep e^{ms^{1+\delta}}-s+m$ on $[0,\infty)$. 
Then we form the system
\begin{eqnarray}
		f_\vep^\prime(s)=\ep m(1+\delta) e^{ms^{1+\delta}}s^\delta-1&=&0,\label{r11}\\
		f_\vep(s)=\ep e^{ms^{1+\delta}}-s+m&=&0\label{r12}
\end{eqnarray}
and solve it for $(\vep,s)$.
By \eqref{r11}, we have
\begin{equation}\label{r14}
	\vep=\frac{1}{m(1+\delta) e^{ms^{1+\delta}}s^\delta}.\nonumber
\end{equation} Substitute this into \eqref{r12} to derive
\begin{equation}\label{r13}
\frac{1}{ m(1+\delta)s^\delta}-s+m=0.
\end{equation}
Clearly, the left-hand side of \eqref{r13} is a strictly decreasing function on $(0,\infty)$ with $f_\ep(0+)=\infty$ and $\lim_{s\rightarrow\infty}f_\vep(s)=-\infty$.  Thus, \eqref{r13} has a unique solution in $(0,\infty)$. This solution is our  $s_0$, and it depends only on $m,\delta$. Use \eqref{epz} to obtain $\ep_0$. Our assumption implies
\begin{equation}\label{f3}
	f_{\vep_0}(a(\tau))> 0\ \ \mbox{for each $\tau\in [0,\tmax]$.}
\end{equation} We easily infer from \eqref{r11} and \eqref{r12} that 
\begin{equation}\label{f4}
f_{\vep_0}(s_0)=f_{\vep_0}^\prime(s_0)=0.
\end{equation}
 Moreover, $f_{\vep_0}(s)$ is strictly decreasing on $[0, s_0)$ and strictly increasing on $(s_0, \infty)$. Thus, $0$ is the minimum value of $f_{\vep_0}(s)$ on $[0,\infty)$.

	It is easy to see from \eqref{r13} that
	\begin{equation}\label{f6}
		s_0>m.
	\end{equation}
	The range of $a$ is a closed interval because it is a continuous function $[0, \tmax]$, and this interval is either contained in $ [0, s_0)$ or $(s_0, \infty)$ due to \eqref{f3} and \eqref{f4}. The latter cannot occur due to the second inequality in \eqref{f5} and \eqref{f6}.  Thus the lemma follows.
	\end{proof}

\section{Approximate problems}
In this section, we first present our approximation scheme. Then we proceed to prove the existence of a weak solution to our approximate problems. Note that our construction here is entirely different from that in \cite{WXZ} where the authors directly seek $C$ in the function space $Z_M\equiv\{u\in C^{\frac{\beta}{2}, \beta}(\overline{\ot}): \frac{1}{M}\leq u\ \mbox{and}\ \|u\|_{C^{\frac{\beta}{2}, \beta}(\overline{\ot})}\leq M\}$ for some $\beta\in(0,1)$ and $M>1$.

To design our approximation scheme,
 we pick
\begin{equation}\label{tdef}
\tau\in\left(0, 1\right)	
\end{equation}
and define
\begin{eqnarray}
	\kappa_\tau(s)&=&\kappa(s^++\tau),\nonumber\\
		\theta_\tau(s)&=&\left\{\begin{array}{ll}
			\frac{1}{\tau}	 &\mbox{if $s\geq\frac{1}{\tau}$,}\\
	s &\mbox{if $\kappa(s)\leq \frac{1}{\tau}$ \ \ \mbox{for $s\in \mathbb{R}$.}}
			\end{array}\right.\label{thdef}
		\end{eqnarray} 
	We approximate $G(y_1,y_2,y_3)$ by
	\begin{equation}\label{gtdf}
	 G_\tau(y_1,y_2,y_3) =y_1(\theta_\tau(y_2)+\tau)^{-d}e^{\alpha_2y_3}-y_1^{-1}\theta_\tau^{d}(y_2)e^{-\alpha_2y_3} \ \ \mbox{for $(y_1,y_2,y_3)\in(0,\infty)\times[0,\infty)\times\mathbb{R}$}.	
	\end{equation}
Our approximate problems are:
\begin{eqnarray}
	-\mdiv\left[\kact\nabla\pe\right]+\tau\pe&=&\frac{1}{2}\alpha_4H_\tau(h,C^+,\ps-\pe)\chi_{\Omega^\prime}\ \ \mbox{in $\ot$},\label{app1}\\
	-\mdiv\left(\sigma\nabla\ps\right)+\tau\ps&=&-\frac{1}{2}\alpha_4H_\tau(h,C^+,\ps-\pe) \ \mbox{in $\opt$},\label{app2}\\
	\ee\pt C-\mdiv\left(D\nabla C\right)&=&\frac{1}{2}\alpha_3\alpha_4H_\tau(h,C^+,\ps-\pe)\chi_{\Omega^\prime}\ \ \mbox{in $\ot$},\label{app3}\\
	\frac{\partial\pe}{\partial \mathbf{n}}&=&0\ \ \mbox{on $\partial\Omega\times(0,T)$},\label{app4}\\
	\frac{\partial \ps}{\partial \mathbf{n}}&=&0\ \ \mbox{on $\partial\op\times(0,T)$},\label{app5}\\
	\frac{\partial C}{\partial \mathbf{n}}&=&0\ \ \mbox{on $\partial\op\times(0,T)$},\label{app55}\\
	C(x,0)&=& C_0(x)\ \ \mbox{on $\Omega$,}\label{app6}
\end{eqnarray} 
where
\begin{equation}\label{htdf}
	H_\tau(y_1,y_2,y_3)=\sqrt{\theta_\tau(y_2)}G_\tau(y_1,y_2,y_3).
\end{equation}

\begin{theorem}\label{thm2}
	Let (H1)-(H7) be satisfied. For each $T>0$ there is a weak solution $(\pe,\ps,C)$ in the space $ X_T\cap Y_T$
	to problem \eqref{app1}-\eqref{app6}, where $ X_T$ and $Y_T$ are defined in \eqref{xdef} and \eqref{ydef}, respectively.
\end{theorem}
The proof of this theorem is divided into a couple of lemmas.
\begin{lemma}
	Let $C=C(x,t)\in  C(\overline{\ot})$ be given.
	For each  $t\in[0,T]$ there is a unique weak solution $(\vp,\psi)=(\vp(x,t),\psi(x,t))$ in $\left(C(\overline{\Omega})\cap W^{1,2}(\Omega)\right)^2$ to the boundary value problem
	\begin{eqnarray}
		-\mdiv\left(\kact\nabla\vp\right)+\tau\vp&=&\frac{1}{2}\alpha_4H_\tau(h,C^+,\psi-\vp)\chi_{\Omega^\prime}\ \ \mbox{in $\Omega$},\label{lsvp1}\\
			-\mdiv\left(\sigma\nabla\psi\right)+\tau\psi&=&-\frac{1}{2}\alpha_4H_\tau(h,C^+,\psi-\vp)\ \ \mbox{in $\Omega^\prime$},\label{lspsi1}\\
			\frac{\partial \vp}{\partial \mathbf{n}}&=& 0\ \ \mbox{on $\po$},\label{lsvp2}\\
		\frac{\partial \psi}{\partial \mathbf{n}}&=& 0\ \ \mbox{on $\partial\Omega^\prime$}.\label{lspsi2}
	\end{eqnarray}
\end{lemma}
\begin{proof}We first establish the uniqueness assertion. Suppose that there exist two weak solutions, say, $(\vp_1,\psi_1), (\vp_2,\psi_2)$. Then we have
	\begin{eqnarray}
		\lefteqn{	-\mdiv\left(\kact\nabla(\vp_1-\vp_2)\right)+\tau(\vp_1-\vp_2)}\nonumber\\
		&=&\frac{1}{2}\alpha_4\sqrt{\theta_\tau(C^+)}\left(G_\tau(h, \thc,\psi_1-\vp_1)-G_\tau(h, \thc,\psi_2-\vp_2)\right)\chi_{\Omega^\prime}\ \ \mbox{in $\Omega$},\label{lsvp3}\\
		\lefteqn{	-\mdiv\left(\sigma\nabla(\psi_1-\psi_2)\right)+\tau(\psi_1-\psi_2)}\nonumber\\
		&=&-\frac{1}{2}\alpha_4\sqrt{\theta_\tau(C^+)}\left(G_\tau(h, \thc,\psi_1-\vp_1)-G_\tau(h, \thc,\psi_2-\vp_2)\right)\ \ \mbox{in $\Omega^\prime$}.\label{lspsi3}
	\end{eqnarray}
Note from \eqref{gtdf} that for each fixed $\tau$ we have
$$G_\tau(h, \thc,\vp)\in L^\infty(\op)\ \ \mbox{whenever $\vp\in L^\infty(\op)$.}$$
Use $\vp_1-\vp_2$ as a test function in \eqref{lsvp3}, $\psi_1-\psi_2$ in \eqref{lspsi3}, then add up the two resulting equations to derive
	\begin{eqnarray}
		\lefteqn{\io \kact|\nabla(\vp_1-\vp_2)|^2dx+\tau\io|\vp_1-\vp_2|^2dx}\nonumber\\
		&&+\iop\sigma|\nabla(\psi_1-\psi_2)|^2dx+\tau\iop|\psi_1-\psi_2|^2dx\nonumber\\
		&=&-\frac{1}{2}\alpha_4\iop \sqrt{\theta_\tau(C^+)}\left(G_\tau(h, \thc,\psi_1-\vp_1)-G_\tau(h, \thc,\psi_2-\vp_2)\right)\left[(\psi_1-\vp_1)-(\psi_2-\vp_2)\right]dx\nonumber\\
		&\leq& 0.\nonumber
	\end{eqnarray}
	The last step is due to \eqref{gpr}.
	Thus,
	$$\psi_1=\psi_2,\ \ \vp_1=\vp_2.$$

	For the existence, we define an operator $\mathbb{T}$ from $\mathbb{B}\equiv L^\infty(\op)\times L^\infty(\Omega)$ into itself as follows: Given $(u, v)\in \mathbb{B}$, we say $\mathbb{T}(u, v)=(\psi,\vp)$ if $\psi,\vp$ are the respective solutions of the following two problems
	\begin{eqnarray}
		-\mdiv\left(\kact\nabla\vp\right)+\tau\vp&=&\frac{1}{2}\alpha_4H_\tau(h, \thc,u-v)\chi_{\Omega^\prime}\ \ \mbox{in $\Omega$},\label{lsvp4}\\
		\frac{\partial \vp}{\partial \mathbf{n}}&=& 0\ \ \mbox{on $\po$},\label{lsvp5}\\
		-\mdiv\left(\sigma\nabla\psi\right)+\tau\psi&=&-\frac{1}{2}\alpha_4H_\tau(h, \thc,u-v)\ \ \mbox{in $\Omega^\prime$},\label{lspsi6}\\
		\frac{\partial \psi}{\partial \mathbf{n}}&=& 0\ \ \mbox{on $\partial\Omega^\prime$}.\label{lspsi7}
	\end{eqnarray}
	Observe from our assumptions that $\kact$ is bounded above and bounded away from $0$ below and the right hand side term in \eqref{lsvp4} lies in $L^\infty(\Omega)$. We can easily conclude from classical theory for linear elliptic equations that problem \eqref{lsvp4}-\eqref{lsvp5}  has a unique weak solution $\vp$  in $W^{1,2}(\Omega)\cap C^\beta(\overline{\Omega})$  for some $\beta\in (0,1)$.  Similarly, problem \eqref{lspsi6}-\eqref{lspsi7} has a unique weak solution $\psi$ in $W^{1,2}(\Omega^\prime)\cap C^\beta(\overline{\Omega^\prime})$.
	Therefore, the operator $\mathbb{T}$ is well-defined. Obviously, any fixed point of $\mathbb{T}$ is a weak solution to problem \eqref{lsvp1}-\eqref{lspsi2}. The existence of such a fixed point will be established by applying Lemma \ref{lsf}, the Leray-Schauder fixed point lemma. To this end,  we can infer from the previously-mentioned classical results that $\mathbb{T}$ is continuous and maps bounded sets into precompact ones. We still need to show that
	there is a positive number $c$ such that
	\begin{equation}\label{vp6}
		\|(v\chi_{\Omega^\prime},u)\|_{\infty,\Omega}\leq c
	\end{equation}
	for all $(u,v)\in \mathbb{B}$ and $\delta\in (0,1)$ satisfying $(u,v)=\delta\mathbb{T}(u,v)$. The preceding equation is equivalent to 
	\begin{eqnarray}
		-\mdiv\left(\kact\nabla u\right)+\tau u&=&\frac{1}{2}\delta \alpha_4H_\tau(h,C^+,v-u)\chi_{\Omega^\prime}\ \ \mbox{in $\Omega$},\label{vp7}\\
		\frac{\partial u}{\partial \mathbf{n}}&=& 0\ \ \mbox{on $\po$},\label{vp8}\\
		-\mdiv\left(\sigma\nabla v\right)+\tau v&=&- \frac{1}{2}\delta\alpha_4H_\tau(h,C^+,v-u)\chi_{\Omega^\prime}\ \ \mbox{in $\Omega^\prime$},\label{psi9}\\
		\frac{\partial v}{\partial \mathbf{n}}&=& 0\ \ \mbox{on $\partial\Omega^\prime$}.\label{psi10}
	\end{eqnarray}
	To see \eqref{vp6}, we use $\left(u^+\right)^n$, where $n\geq 1$, as a test function in \eqref{vp7} and keep in mind \eqref{vp8} to get
	\begin{equation}\label{sm1}
		\tau\io \left(u^+\right)^{n+1}dx\leq \frac{1}{2}\delta \alpha_4\iop H_\tau(h,C^+,v-u)\left(u^+\right)^ndx.
	\end{equation}
	Similarly, we can infer from \eqref{psi9} and \eqref{psi10} that
	\begin{equation}
		\tau\iop\left(v^+\right)^{n+1} dx\leq -\frac{1}{2}\delta \alpha_4\iop H_\tau(h,C^+,v-u)\left(v^+\right)^ndx.\nonumber
	\end{equation}
	Add this to \eqref{sm1} to derive
	\begin{eqnarray}
		\lefteqn{	\tau\io\left(u^+\right)^{n+1}dx+\tau\iop\left(v^+\right)^{n+1}  dx}\nonumber\\
		&\leq&-\frac{1}{2}\delta\alpha_4\iop H_\tau(h, \thc, v- u) \left( \left(v^+\right)^n-\left(u^+\right)^n\right)dx\nonumber\\
		&=&-\frac{1}{2}\delta\alpha_4\iop \sqrt{\theta_\tau(C^+)}\left[G_\tau(h, \thc, v- u)-G_\tau(h, \thc,0)\right] \left( \left(v^+\right)^n-\left(u^+\right)^n\right)dx\nonumber\\
		&&+\frac{1}{2}\delta\alpha_4\iop H_\tau(h, \thc,0)\left( \left(v^+\right)^n-\left(u^+\right)^n\right)dx.\label{uve2}
	\end{eqnarray}
	Recall from \eqref{gpr} that $G_\tau(y_1,y_2,y_3)$ is increasing in $y_3$. Thus,
	\begin{equation}
		\left[G_\tau(h, \thc, v- u)-G_\tau(h, \thc,0)\right] \left( \left(v^+\right)^n-\left(u^+\right)^n\right)\geq 0\ \ \mbox{in $\op$.}\nonumber
	\end{equation}
	Use this in \eqref{uve2} to derive
	\begin{eqnarray}
		\lefteqn{	\tau\io\left(u^+\right)^{n+1}dx+\tau\iop\left(v^+\right)^{n+1}  dx}\nonumber\\
		&\leq&\frac{1}{2}\alpha_4\iop \sqrt{\theta_\tau(C^+)}\left|G_\tau(h, \thc,0)\right|\left( \left(v^+\right)^n+\left(u^+\right)^n\right)dx\nonumber\\
		&\leq& \alpha_4\iop \sqrt{\theta_\tau(C^+)}\left|G_\tau(h, \thc,0)\right|w^ndx,\nonumber
	\end{eqnarray}
	where
	\begin{equation}
		w=\max\{u^+, v^+\chi_{\Omega^\prime}\}.\nonumber
	\end{equation}
	Consequently,
	\begin{eqnarray}
		\tau\io w^{n+1}dx&\leq &\tau\io\left(u^+\right)^{n+1}dx+\tau\iop\left(v^+\right)^{n+1}  dx\nonumber\\
		&\leq&\alpha_4\iop \sqrt{\theta_\tau(C^+)}\left|G_\tau(h, \thc,0)\right|w^ndx\nonumber\\
		&\leq& \alpha_4\|w\|_{n+1,\Omega}^n\|H_\tau(h, \thc,0)\|_{n+1,\Omega},\nonumber
	\end{eqnarray}
	from whence follows
	$$\|w\|_{n+1,\Omega}\leq \frac{\alpha_4}{\tau}\|H_\tau(h, \thc,0)\|_{n+1,\Omega}.$$
	Taking $n\ra\infty$ yields 
	\begin{equation}\label{uve11}
		\max\{\|u^+\|_{\infty,\Omega}, \|v^+\|_{\infty,\op}\}\leq \frac{\alpha_4}{\tau}\|H_\tau(h, \thc,0)\|_{\infty,\Omega}.
	\end{equation}
	In an entirely similar manner, we can also prove
	\begin{equation}\label{uve12}
		\max\{\|u^-\|_{\infty,\Omega}, \|v^-\|_{\infty,\op}\}\leq \frac{\alpha_4}{\tau}\|H_\tau(h, \thc,0)\|_{\infty,\Omega}.
	\end{equation}
	Combining this with \eqref{uve11} yields \eqref{vp6}. The proof is complete.
\end{proof}

To continue the proof of Theorem \ref{thm2}, we define an operator $\mathbb{A}$ from $ C(\overline{\ot})$ into $ C(\overline{\ot})$ as follows: For each $C\in  C(\overline{\ot})$, we first solve system \eqref{app1}-\eqref{app2} coupled with boundary conditions \eqref{app4} and \eqref{app5} for $(\pe,\ps)$. Then use the pair to form the problem
\begin{eqnarray}
	\ee\pt B-\mdiv\left(D\nabla  B\right)&=& \frac{1}{2}\alpha_3\alpha_4 H_\tau(h,C^+,\ps-\pe)\ \ \mbox{in $\ot$,}\label{b1}\\
	\frac{\partial B}{\partial \mathbf{n}}&=& 0\ \ \mbox{on $\po\times(0,T)$,}\label{b2}\\
	B(x,0)&=& C_0(x)\ \ \mbox{on $\Omega$.}\label{b3}
\end{eqnarray}
 Classical regularity for linear parabolic equations (\cite{LSU}, Chap.III) asserts that problem \eqref{b1}-\eqref{b3} has a unique weak solution $B$ in the space $C([0,T]; L^2(\Omega))\cap L^2(0,T; W^{1,2}(\Omega))$. Moreover, $B$ is H\"{o}lder continuous on $\overline{\ot}$. We define  $B=\mathbb{A}(C)$. It is easy to see that $\mathbb{A}$ is well-defined and continuous and maps bounded sets into precompact ones. To be able to use Lemma \ref{lsf},  we still need to show that 
\begin{equation}\label{cbd2}
	\|C\|_{\infty,\ot}\leq c
\end{equation}
 for all $\delta\in [0,1]$ and $C\in C(\overline{\ot})$ satisfying $C=\delta \mathbb{A}(C)$. This equation is equivalent to the problem
\begin{eqnarray}
	-\mdiv\left[\kact\nabla\pe\right]+\tau\pe&=&\frac{1}{2}\alpha_4H_\tau(h,C^+,\ps-\pe)\chi_{\Omega^\prime}\ \ \mbox{in $\ot$},\label{lapp1}\\
	-\mdiv\left(\sigma\nabla\ps\right)+\tau\ps&=&-\frac{1}{2}\alpha_4H_\tau(h,C^+,\ps-\pe) \ \mbox{in $\opt$},\label{lapp2}\\
	\ee\pt C-\mdiv\left(D\nabla C\right)&=&\frac{1}{2}\alpha_3\alpha_4\delta H_\tau(h,C^+,\ps-\pe)\chi_{\Omega^\prime}\ \ \mbox{in $\ot$},\label{lapp3}\\
	\frac{\partial\pe}{\partial \mathbf{n}}=\frac{\partial C}{\partial \mathbf{n}}&=&0\ \ \mbox{on $\partial\Omega\times(0,T)$},\label{lapp4}\\
	\frac{\partial \ps}{\partial \mathbf{n}}&=&0\ \ \mbox{on $\partial\op\times(0,T)$},\label{lapp5}\\
	C(x,0)&=& \delta C_0(x)\ \ \mbox{on $\Omega$.}\label{lapp6}
\end{eqnarray}

	To see \eqref{cbd2}, we use $C^-$ as a test function in \eqref{lapp3} to get
	$$-\frac{1}{2}\frac{d}{dt}\io\ee(C^-)^2dx-\io D|\nabla C^-|^2dx=0,$$
	from whence follows
	\begin{equation}
		C\geq 0\ \ \mbox{a.e. on $\ot$.}\nonumber
	\end{equation}
Thus, we can replace $C^+$ in \eqref{lapp1}-\eqref{lapp3} by $C$. We will do this in the subsequent calculations without acknowledgment.

An upper bound for $C$ is established in the following lemma.
\begin{lemma}\label{lcub}For each $q>1+\frac{N}{2}$ there is a positive number $c$ depending on $N, q, \Omega$ and the given data in system \eqref{cpe}-\eqref{cc} such that
	\begin{equation}
		\max_{\ot} C\leq 	c(1+T)^{\frac{2N}{(N+2)}}T^{\frac{2(2q-2-N)}{q(N+2)}}\left\|e^{\alpha_2(\ps-\pe)}\right\|_{q,\opt}^{2}+4\|C_0\|_{\infty,\Omega}+2.\nonumber
	\end{equation}
\end{lemma}
For our later purpose, the number $c$ in the above inequality will be made independent of $\tau$.
\begin{proof} We employ the well known De Giorgi iteration scheme \cite{D}. Let $q$ be given as in the lemma. Select
	\begin{equation}\label{klb}
		k\geq \max\left\{2\|C_0\|_{\infty,\Omega},1\right\}
	\end{equation} as below. Define
	$$k_n=k-\frac{k}{2^{n+1}},\ \ n=0,1,\cdots.$$
	Consequently, 
	\begin{equation}\label{conk}
		k\geq k_{n+1}\geq	k_n\geq\frac{k}{2}\geq \|C_0\|_{\infty,\Omega} \ \ \mbox{for each $n\in\{0,1,\cdots\}$}.
	\end{equation}
	We use $(C-k_{n+1})^+$ as a test function in \eqref{lapp3} to derive
	\begin{eqnarray}\label{cub1}
		\lefteqn{	\frac{1}{2}\frac{d}{dt}\io \varepsilon_e[(C-k_{n+1})^+]^2dx+\io D|\nabla(C-k_{n+1})^+|^2dx}\nonumber\\
		&=&\frac{1}{2}\alpha_3\alpha_4\delta\iop H_\tau(h,C,\ps-\pe)(C-k_{n+1})^+dx.
	\end{eqnarray}
	Note that $G_\tau(y_1,y_2,y_3)$ is decreasing in $y_2$. Thus,
	$$\left[G_\tau(h,C,\ps-\pe)-G_\tau(h,k_{n+1},\ps-\pe)\right](C-k_{n+1})^+\leq 0.$$
	With this in mind, we can estimate the last integral in \eqref{cub1} as follows:
	\begin{eqnarray}
		\lefteqn{\iop H_\tau(h,C,\ps-\pe)(C-k_{n+1})^+dx}\nonumber\\
		&=&\iop\sqrt{\theta_\tau(C)} \left(G_\tau(h,C,\ps-\pe)-G_\tau(h,k_{n+1},\ps-\pe)\right)(C-k_{n+1})^+dx\nonumber\\
		&&+\iop H_\tau(h,k_{n+1},\ps-\pe)(C-k_{n+1})^+dx\nonumber\\
		&\leq&\frac{\left\|h\right\|_{\infty,\opt}}{(\theta_\tau(k_{n+1})+\tau)^{d}}\iop \sqrt{\theta_\tau(C)} e^{\alpha_2(\ps-\pe)}(C-k_{n+1})^+dx.\label{dt1}
	\end{eqnarray}
	Here we have used \eqref{gtdf},  the definition of $G_\tau$. We easily verify from \eqref{tdef}, \eqref{thdef}, \eqref{klb}, and \eqref{conk} that
	$$\frac{1}{\theta_\tau(k_{n+1})+\tau}\leq 2.$$
	Substitute \eqref{dt1} into \eqref{cub1},  integrate the resulting inequality with respect to $t$, and keep in mind the fact that $\left.(C-k_{n+1})^+\right|_{t=0}=0$ due to \eqref{conk} to deduce
	\begin{eqnarray}
		\lefteqn{\sup_{0\leq t\leq T}\io \ee[(C-k_{n+1})^+]^2dx+\ioT D|\nabla(C-k_{n+1})^+|^2dxdt}\nonumber\\
		&\leq&c\int_{\opt} \sqrt{\theta_\tau(C)} e^{\alpha_2(\ps-\pe)}(C-k_{n+1})^+dxdt.\label{ce1}
	\end{eqnarray}
	Note that
	\begin{eqnarray}
		\left[(C-k_{n})^+\right]^{\frac{3}{2}}&\geq&\sqrt{(C-k_{n})^+}(C-k_{n+1})^+\nonumber\\
		&\geq & \sqrt{C\left(1-\frac{k_n}{k_{n+1}}\right)}(C-k_{n+1})^+\nonumber\\
		&\geq &\frac{1}{2^{\frac{n+2}{2}}}\sqrt{\theta_\tau(C)}(C-k_{n+1})^+,\nonumber
	\end{eqnarray}
from whence follows
	\begin{eqnarray}
		\int_{\opt} \sqrt{\theta_\tau(C)} e^{\alpha_2(\ps-\pe)}(C-k_{n+1})^+dxdt	&\leq&2^{\frac{n+2}{2}} \int_{\opt} e^{\alpha_2(\ps-\pe)} \left[(C-k_{n})^+\right]^{\frac{3}{2}}dxdt\nonumber\\
		&\leq& 2^{\frac{n+2}{2}}\|e^{\alpha_2(\ps-\pe)}\|_{q,\opt}y_n^{\frac{q-1}{q}},\label{u11}
	\end{eqnarray}
	where
	$$y_n=\ioT\left[(C-k_{n})^+\right]^{\frac{3q}{2(q-1)}}dxdt.$$
	Recall from the Sobolev embedding theorem that
	\begin{eqnarray}
		\lefteqn{\ioT\left[(C-k_{n+1})^+\right]^{2+\frac{4}{N}}dxdt}\nonumber\\
		&\leq&\int_{0}^{T}\left(\io\left[(C-k_{n+1})^+\right]^2dx \right)^{\frac{2}{N}}\left(\io\left[(C-k_{n+1})^+\right]^{\frac{2N}{N-2}}dx \right)^{\frac{N-2}{N}}dt\nonumber\\
		&\leq&c\left(\sup_{0\leq t\leq T}\io\left[(C-k_{n+1})^+\right]^2dx \right)^{\frac{2}{N}}\ioT\left|\nabla(C-k_{n+1})^+\right|^2dxdt\nonumber\\
		&&+c\left(\sup_{0\leq t\leq T}\io\left[(C-k_{n+1})^+\right]^2dx \right)^{\frac{2}{N}}\int_{0}^{T}\io\left[(C-k_{n+1})^+\right]^2dxdt\nonumber\\
		&\leq&c(1+T)2^{\frac{(n+2)(N+2)}{2N}}\left\|e^{\alpha_2(\ps-\pe)}\right\|_{q,\opt}^{\frac{N+2}{N}}y_n^{\frac{(q-1)(N+2)}{qN}}
		.\label{ce3}
	\end{eqnarray}
	The last step is due to \eqref{ce1} and \eqref{u11}. We easily see that 
	$$y_n\geq \int_{\ot\cap\{C\geq k_{n+1}\}}(k_{n+1}-k_n)^{\frac{3q}{2(q-1)}}dxdt=\frac{k^{\frac{3q}{2(q-1)}}}{2^{\frac{3q(n+2)}{2(q-1)}}}\left|\{C\geq k_{n+1}\}\right|.$$
	Furthermore, our assumption on $q$ implies that
	\begin{equation}
		2+\frac{4}{N}>\frac{3q}{2(q-1)},\ \ \beta\equiv\frac{3}{4}-\frac{3qN}{4(N+2)(q-1)}>0.\nonumber
	\end{equation}
	Keeping the preceding three inequalities  and \eqref{ce3} in mind, we estimate that
	\begin{eqnarray}
		y_{n+1}&\leq&\left(\ioT\left[(C-k_{n+1})^+\right]^{2+\frac{4}{N}}dxdt\right)^{\frac{3qN}{4(N+2)(q-1)}}\left|\{C\geq k_{n+1}\}\right|^{\frac{1+4\beta}{4}}\nonumber\\
		&\leq&\frac{c(1+T)^{\frac{3qN}{4(N+2)(q-1)}}\left\|e^{\alpha_2(\ps-\pe)}\right\|_{q,\opt}^{\frac{3q}{4(q-1)}}2^{\frac{3q(n+2)}{8(q-1)}\left(\frac{1}{2}+\beta\right)}}{k^{^{\frac{3q(1+4\beta)}{8(q-1)}}}}y_n^{1+\beta}.\nonumber
	\end{eqnarray}
	By Lemma \ref{ynb}, if we choose $k$ so large that 
	$$y_0\leq\ioT C^{\frac{3q}{2(q-1)}}dxdt\leq \frac{ck^{^{\frac{3q(1+4\beta)}{8\beta(q-1)}}}}{(1+T)^{\frac{3qN}{4\beta(N+2)(q-1)}}\left\|e^{\alpha_2(\ps-\pe)}\right\|_{q,\opt}^{\frac{3q}{4(q-1)\beta}}}
	$$	then
	$$C\leq k.$$ In view of \eqref{klb}, it is enough for us to take
	$$k=c(1+T)^{\frac{2N}{(N+2)(1+4\beta)}}\left\|e^{\alpha_2(\ps-\pe)}\right\|_{q,\opt}^{\frac{2}{1+4\beta}}\|C\|_{\frac{3q}{2(q-1)},\ot}^{\frac{4\beta}{1+4\beta}}+2\|C_0\|_{\infty,\Omega}+1.$$
	That is,
	\begin{eqnarray}
		\|C\|_{\infty,\ot}&\leq&c(1+T)^{\frac{2N}{(N+2)(1+4\beta)}}\left\|e^{\alpha_2(\ps-\pe)}\right\|_{q,\opt}^{\frac{2}{1+4\beta}}\|C\|_{\frac{3q}{2(q-1)},\ot}^{\frac{4\beta}{1+4\beta}}+2\|C_0\|_{\infty,\Omega}+1\nonumber\\
		&\leq&c(1+T)^{\frac{2N}{(N+2)(1+4\beta)}}T^{\frac{8\beta(q-1)}{3q(1+4\beta)}}\left\|e^{\alpha_2(\ps-\pe)}\right\|_{q,\opt}^{\frac{2}{1+4\beta}}\|C\|_{\infty,\ot}^{\frac{4\beta}{1+4\beta}}
		+2\|C_0\|_{\infty,\Omega}+1\nonumber\\
		&\leq&\frac{1}{2}	\|C\|_{\infty,\ot}+c(1+T)^{\frac{2N}{N+2}}T^{\frac{8\beta(q-1)}{3q}}\left\|e^{\alpha_2(\ps-\pe)}\right\|_{q,\opt}^{2}+2\|C_0\|_{\infty,\Omega}+1.\nonumber
	\end{eqnarray}
	This implies the desired result.\end{proof}
To conclude the proof of Theorem \ref{thm2},  we can easily infer from \eqref{uve11}, \eqref{uve12}, and \eqref{htdf} that
$$\|\pe\|_{\infty,\ot}+\|\ps\|_{\infty,\opt}\leq c(\tau).$$
This combined with Lemma \ref{lcub} yields \eqref{cbd2}. Now we are ready to invoke Lemma \ref{lsf}. Upon doing so, we obtain a fixed point of $\mathbb{A}$, which produces a weak solution to problem \eqref{app1}-\eqref{app6}. This completes the proof of Theorem \ref{thm2}.

\section{Proof of Theorem \ref{thm1}}
In this section, we offer the proof of Theorem \ref{thm1}. For this purpose, we denote the solutions to problem \eqref{app1}-\eqref{app6} constructed in Theorem \ref{thm2} by $(\pet,\pst,\ct)$ with $\tau=\frac{1}{k}$, where $k=1,2,\cdots$. The plan is to derive enough a priori estimates for the sequence to justify passing to the limit in problem \eqref{app1}-\eqref{app6}. The generic positive constant $c$ in the subsequent calculations will be independent of $\tau$. The key to our development is the following lemmas, which asserts that $\ct$ is bounded away from $0$ below.
\begin{lemma}\label{lclb}
	For each $q>1+\frac{N}{2}$ there is a positive number $c$ such that
	\begin{equation}
		\frac{1}{\ct}\leq \frac{1}{\min_{\overline{\Omega}}C_0}+ c(1+T)^{\frac{2N}{(N+2)(2d-1)}}\left\|e^{-\alpha_2(\pst-\pet)}\right\|_{q,\opt}^{\frac{2}{2d-1}}.\nonumber
	\end{equation}
\end{lemma}
\begin{proof}
	Pick 
	\begin{equation}\label{kcon2}
		\frac{2}{k}\leq\min_{\overline{\Omega}}C_0
	\end{equation}
	as below.
	Set
	$$k_n=k-\frac{k}{2^{n+1}},\ \ n=0,1,\cdots.$$
	Recall that $G_\tau(y_1,y_2,y_3)$ is decreasing in $y_2$. Thus,
	$$M_\tau\equiv \left[G_\tau(h,\ct,\pst-\pet)-G_\tau(h,k_{n+1}^{-1},\pst-\pet)\right]\left(\left(\ct\right)^{-\frac{1}{2}}-k_{n+1}^{\frac{1}{2}}\right)^+\geq 0.$$
	Use $\left(\left(\ct\right)^{-\frac{1}{2}}-k_{n+1}^{\frac{1}{2}}\right)^+$ as a test function in \eqref{lapp3} and apply the above inequality and the definition of $G_\tau$ to get
	\begin{eqnarray}
		\lefteqn{\frac{d}{dt}\io\ee\int_{\ct}^{k_{n+1}^{-1}}\left(s^{-\frac{1}{2}}-k_{n+1}^{\frac{1}{2}}\right)^+dsdx+8\io D\left|\nabla\left(k_{n+1}^{-\frac{1}{4}}-\left(\ct\right)^{\frac{1}{4}}\right)^+\right|^2dx}\nonumber\\
		&=&-\frac{1}{2}\alpha_3\alpha_4\delta\iop H_\tau(h,\ct,\pst-\pet)\left(\left(\ct\right)^{-\frac{1}{2}}-k_{n+1}^{\frac{1}{2}}\right)^+dx\nonumber\\
		&=&-\frac{1}{2}\alpha_3\alpha_4\delta\iop\sqrt{\theta_\tau(\ct)}M_\tau dx\nonumber\\
		&&-\frac{1}{2}\alpha_3\alpha_4\delta\iop \sqrt{\theta_\tau(\ct)} G_\tau(h,k_{n+1}^{-1},\pst-\pet)\left(\left(\ct\right)^{-\frac{1}{2}}-k_{n+1}^{\frac{1}{2}}\right)^+dx\nonumber\\
		&\leq &\frac{1}{2}\alpha_3\alpha_4\delta  \theta_\tau^d\left(k_{n+1}^{-1}\right)\int_{\Omega^\prime\cap\{\ct\leq k_{n+1}^{-1}\}} h^{-1}e^{-\alpha_2(\pst-\pet)}dx\nonumber\\
		&\leq &\frac{c}{k^d}\int_{\Omega^\prime\cap\{\ct\leq k_{n+1}^{-1}\}}e^{-\alpha_2(\pst-\pet)}dx.\nonumber
	\end{eqnarray}
	The last step is due to \eqref{kcon2} and \eqref{hb}.
	Using \eqref{kcon2} again yields $\left.\left(\left(\ct\right)^{-\frac{1}{2}}-k_{n+1}^{\frac{1}{2}}\right)^+\right|_{t=0}=0$. With this in mind, we integrate the above inequality to derive
	\begin{eqnarray}
		\lefteqn{\sup_{0\leq t\leq T}\io\ee\int_{\ct}^{k_{n+1}^{-1}}\left(s^{-\frac{1}{2}}-k_{n+1}^{\frac{1}{2}}\right)^+ds dx}\nonumber\\
		&&+\ioT D\left|\nabla\left(k_{n+1}^{-\frac{1}{4}}-\left(\ct\right)^{\frac{1}{4}}\right)^+\right|^2dxdt\nonumber\\
		&\leq &\frac{c}{k^d}\int_{\Omega^\prime\cap\{\ct\leq k_{n+1}^{-1}\}}e^{-\alpha_2(\pst-\pet)}dxdt\nonumber\\
		&\leq &\frac{c}{k^d}\|e^{-\alpha_2(\pst-\pet)}\|_{q,\opt}\left|\opt\cap\left\{\ct\leq k_{n+1}^{-1}\right\}\right|^{\frac{q-1}{q}}.\label{uce4}
	\end{eqnarray}
	Note that on the set $\ot\cap\left\{\ct\leq k_{n+1}^{-1}\right\}$ we have
	\begin{eqnarray}
		\int_{\ct}^{k_{n+1}^{-1}}\left(\frac{1}{\sqrt{s}}-\sqrt{k_{n+1}}\right)^+ds&=&k_{n+1}^{-\frac{1}{2}}-2\sqrt{\theta_\tau(\ct)}+\sqrt{k_{n+1}}\ct\nonumber\\
		&=&\left(k_{n+1}^{-\frac{1}{4}}-(k_{n+1}\ct)^{\frac{1}{4}}\left(\ct\right)^{\frac{1}{4}}\right)^2\nonumber\\
		&\geq&\left[\left(k_{n+1}^{-\frac{1}{4}}-\left(\ct\right)^{\frac{1}{4}}\right)^+\right]^2.\label{uce5}
	\end{eqnarray}
	Recall from the Sobolev embedding theorem that
	\begin{eqnarray}
		y_{n+1}&\equiv&\ioT\left[\left(k_{n+1}^{-\frac{1}{4}}-\left(\ct\right)^{\frac{1}{4}}\right)^+\right]^{2+\frac{4}{N}}dxdt\nonumber\\
		&\leq&\int_{0}^{T}\left(\io\left[\left(k_{n+1}^{-\frac{1}{4}}-\left(\ct\right)^{\frac{1}{4}}\right)^+\right]^2dx \right)^{\frac{2}{N}}\left(\io\left[\left(k_{n+1}^{-\frac{1}{4}}-\left(\ct\right)^{\frac{1}{4}}\right)^+\right]^{\frac{2N}{N-2}}dx \right)^{\frac{N-2}{N}}dt\nonumber\\
		&\leq&c\left(\sup_{0\leq t\leq T}\io\left[\left(k_{n+1}^{-\frac{1}{4}}-\left(\ct\right)^{\frac{1}{4}}\right)^+\right]^2dx \right)^{\frac{2}{N}}\ioT\left|\nabla\left(k_{n+1}^{-\frac{1}{4}}-\left(\ct\right)^{\frac{1}{4}}\right)^+\right|^2dxdt\nonumber\\
		&&+c\left(\sup_{0\leq t\leq T}\io\left[\left(k_{n+1}^{-\frac{1}{4}}-\left(\ct\right)^{\frac{1}{4}}\right)^+\right]^2dx \right)^{\frac{2}{N}}\int_{0}^{T}\io\left[\left(k_{n+1}^{-\frac{1}{4}}-\left(\ct\right)^{\frac{1}{4}}\right)^+\right]^2dxdt\nonumber\\
		&\leq&\frac{c(1+T)}{k^{\frac{d(N+2)}{N}}}\left\|e^{-\alpha_2(\pst-\pet)}\right\|_{q,\opt}^{\frac{N+2}{N}}\left|\opt\cap\left\{\ct\leq k_{n+1}^{-1}\right\}\right|^{\frac{(q-1)(N+2)}{qN}}.
		.\label{uce3}
	\end{eqnarray}
	The last step is due to \eqref{uce5} and \eqref{uce4}.
	Observe that 
	\begin{eqnarray}
		y_n&\geq&\int_{\left\{\ct\leq k_{n+1}^{-1}\right\}}\left[\left(k_{n}^{-\frac{1}{4}}-\left(\ct\right)^{\frac{1}{4}}\right)^+\right]^{2+\frac{4}{N}}dxdt\nonumber\\
		&\geq&\left(k_n^{-\frac{1}{4}}-k_{n+1}^{-\frac{1}{4}}\right)^{2+\frac{4}{N}}\left|\opt\cap\left\{\ct\leq k_{n+1}^{-1}\right\}\right|\nonumber\\
		&\geq&\left(\frac{k_{n+1}-k_n}{4k_n^{\frac{1}{4}}k_{n+1}}\right)^{2+\frac{4}{N}}\left|\opt\cap\left\{\ct\leq k_{n+1}^{-1}\right\}\right|\nonumber\\
		&\geq&\frac{c}{2^{\frac{2(n+2)(N+2)}{N}}k^{\frac{N+2}{2N}}}\left|\opt\cap\left\{\ct\leq k_{n+1}^{-1}\right\}\right|.\nonumber
	\end{eqnarray}
	Use this in \eqref{uce3} to get
	\begin{equation}
		y_{n+1}\leq \frac{c(1+T)}{k^{\frac{(N+2)}{N}\left(d-\frac{1+\lambda}{2}\right)}}\left\|e^{-\alpha_2(\pst-\pet)}\right\|_{q,\opt}^{\frac{N+2}{N}}2^{\frac{2(n+2)(N+2)(1+\lambda)}{N}}y_n^{1+\lambda},\nonumber
	\end{equation}
	where
	\begin{equation}
		\lambda=\frac{(q-1)(N+2)}{qN}-1>0\nonumber
	\end{equation}
	due to our assumption.
	According to Lemma \ref{ynb}, if we choose $k$ so large that 
	\begin{eqnarray}
		y_0&=&\ioT\left[\left(\left(\frac{k}{2}\right)^{-\frac{1}{4}}-\left(\ct\right)^{\frac{1}{4}}\right)^+\right]^{2+\frac{4}{N}}dxdt\nonumber\\
		&\leq& \left(\frac{k}{2}\right)^{-\frac{N+2}{2N}}|\ot|\leq \frac{k^{\frac{(N+2)}{N\lambda}\left(d-\frac{1+\lambda}{2}\right)}}{c(1+T)^{\frac{1}{\lambda}}\left\|e^{-\alpha_2(\pst-\pet)}\right\|_{q,\opt}^{\frac{N+2}{N\lambda}}}	\label{uce7}
	\end{eqnarray}
	then
	$$\ct\geq \frac{1}{k}.$$ Under (H8), \eqref{uce7} is equivalent to
	$$k\geq c(1+T)^{\frac{2N}{(N+2)(2d-1)}}\left\|e^{-\alpha_2(\pst-\pet)}\right\|_{q,\opt}^{\frac{2}{2d-1}}.$$
	This combined with \eqref{kcon2} implies the desired result. The proof is complete.
\end{proof}
\begin{lemma}We have
	\begin{equation}\label{esl2}
			\io\left(\left(\pet\right)^2+\left(\pst\right)^2\chi_{\Omega^\prime}\right)dx
		\leq c\left(\lt\right)^{\frac{3d+3+2\alpha_0}{2}}\iop H_\tau^2(h,\ct,0)dx,
	\end{equation}
where $\alpha_0$ is given as in \eqref{kacon} and 
\begin{equation}
	\lt=\lt(T)=\max_{\overline{\ot}}\frac{1}{\ct}.\nonumber
\end{equation}
	\end{lemma}
\begin{proof}	Integrate \eqref{lapp1} over $\Omega$ and take into account the boundary condition \eqref{lapp4} to get
	\begin{equation}\label{es1}
		\tau\io\pet dx=\frac{1}{2}\alpha_4\iop H_\tau(h,\ct,\pst-\pet)dx.
	\end{equation}
In the same manner, we can deduce from \eqref{lapp2} and \eqref{lapp5} that
$$\tau\iop\pst dx=-\frac{1}{2}\alpha_4\iop H_\tau(h,\ct,\pst-\pet)dx.$$
Add this to \eqref{es1} to obtain
\begin{equation}\label{es2}
	\io\pet dx+\iop\pst dx=0.
\end{equation}
Use $\pet$ as a test function in \eqref{lapp1}, then $\pst$ in \eqref{lapp2}, and add the two resulting equations to derive
\begin{eqnarray}
	\lefteqn{	\io\kactt|\nabla\pet|^2dx+\iop\sigma|\nabla\pst|^2dx+\tau\io\left(\pet\right)^2dx+\tau\iop\left(\pst\right)^2dx}\nonumber\\
	&&+\frac{1}{2}\alpha_4\iop H_\tau(h,\ct,\pst-\pet) (\pst-\pet)dx=0.\label{rej1}
\end{eqnarray}
The mean value theorem asserts that there is an $\xi$ lying between $0$ and $\pst-\pet$ such that
\begin{eqnarray}
	G_\tau(h,\ct,\pst-\pet)&=&G_\tau(h,\ct,\pst-\pet)-G_\tau(h,\ct, 0)+G_\tau(h,\ct ,0)\nonumber\\
	&=&\partial_{y_3}G_{\tau}(h,\ct,\xi)(\pst-\pet)+G_\tau(h,\ct,0).\nonumber
\end{eqnarray}
Note from \eqref{gpr} that
\begin{eqnarray}
	\partial_{y_3}G_{\tau}(h,\ct,\xi)	&=&\alpha_2\left(h(\theta_\tau(\ct)+\tau)^{-d}e^{\alpha_2\xi}+h^{-1}\theta_\tau^d(\ct)e^{-\alpha_2\xi}\right)\nonumber\\
	&\geq &\frac{2\alpha_2\theta_\tau^{\frac{d}{2}}(\ct)}{(\theta_\tau(\ct)+\tau)^{\frac{d}{2}}}.\label{rej4}
\end{eqnarray}
With these in mind, we can estimate the last integral in \eqref{rej1} as follows:
\begin{eqnarray}
	\lefteqn{\frac{1}{2}\alpha_4\iop H_\tau(h,\ct,\pst-\pet) (\pst-\pet)dx}\nonumber\\
	&=&\frac{1}{2}\alpha_4\iop \sqrt{\theta_\tau(\ct)}\partial_{y_3}G(h,\ct,\xi)(\pst-\pet)^2dx\nonumber\\
	&&+\frac{1}{2}\alpha_4\iop \sqrt{\theta_\tau(\ct)} G_\tau(h,\ct,0)(\pst-\pet)dx\nonumber\\
&\geq& \alpha_2\alpha_4\iop \frac{\theta_\tau^{\frac{d+1}{2}}(\ct)}{(\theta_\tau(\ct)+\tau)^{\frac{d}{2}}}(\pst-\pet)^2dx+\frac{1}{2}\alpha_4\iop H_\tau(h,\ct,0)(\pst-\pet)dx\nonumber\\
	&\geq& \frac{1}{2}\alpha_2\alpha_4\iop \frac{\theta_\tau^{\frac{d+1}{2}}(\ct)}{(\theta_\tau(\ct)+\tau)^{\frac{d}{2}}}(\pst-\pet)^2dx-\frac{\alpha_4}{8\alpha_2}\iop \frac{(\theta_\tau(\ct)+\tau)^{\frac{d}{2}}}{\theta_\tau^{\frac{d-1}{2}}(\ct)}G_\tau^2(h,\ct,0)dx.\nonumber
\end{eqnarray}
Plug this into \eqref{rej1} to derive
\begin{eqnarray}
	\lefteqn{\io \kactt|\nabla\pet|^2dx+\iop\sigma|\nabla\pst|^2dx+\tau\io\left(\pet\right)^2dx+\tau\iop\left(\pst\right)^2dx}\nonumber\\
	&&+ \frac{1}{2}\alpha_2\alpha_4\iop \frac{\theta_\tau^{\frac{d+1}{2}}(\ct)}{(\theta_\tau(\ct)+\tau)^{\frac{d}{2}}}(\pst-\pet)^2dx\leq \frac{\alpha_4}{8\alpha_2}\iop \frac{(\theta_\tau(\ct)+\tau)^{\frac{d}{2}}}{\theta_\tau^{\frac{d-1}{2}}(\ct)} G_\tau^2(h,\ct,0)dx.\label{rej2}
\end{eqnarray}
Note that the function $\frac{s^{\frac{d+1}{2}}}{(s+\tau)^{\frac{d}{2}}}$ is increasing on $[0,\infty)$. Thus,
$$ \frac{\theta_\tau^{\frac{d+1}{2}}(\ct)}{(\theta_\tau(\ct)+\tau)^{\frac{d}{2}}}\geq  \frac{\theta_\tau^{\frac{d+1}{2}}(m_c)}{(\theta_\tau(m_c)+\tau)^{\frac{d}{2}}},\ \ \mbox{where $m_c=\min_{\overline{\ot}}\ct=\frac{1}{\lt}$}.$$
We derive from \eqref{rej2} that
\begin{eqnarray}
		\iop (\pst-\pet)^2dx&\leq &\frac{(\theta_\tau(m_c)+\tau)^{\frac{d}{2}}}{4\alpha_2^2\theta_\tau^{\frac{d+1}{2}}(m_c)}\iop \frac{(\theta_\tau(\ct)+\tau)^{\frac{d}{2}}}{\theta_\tau^{\frac{d-1}{2}}(\ct)}G_\tau^2(h,\ct,0)dx\nonumber\\
	&\leq &\frac{(\theta_\tau(m_c)+\tau)^{d}}{4\alpha_2^2\theta_\tau^{d+1}(m_c)}\iop H_\tau^2(h,\ct,0)dx\nonumber\\
		&\leq &c\left(\lt\right)^{d+1}\iop H_\tau^2(h,\ct,0)dx.\label{ese4}
\end{eqnarray}
Here we have used the obvious fact that $m_c\leq \min_{\overline{\Omega}} C_0(x)$ and assumed $\frac{1}{\tau}\geq \min_{\overline{\Omega}} C_0(x)$. In view of \eqref{kacon}, we have
$$\kactt\geq \min\{c_0(\ct+\tau)^{\alpha_0},\min_{s\in[s_0,\infty)}\kappa(s)\}\geq \min\{c_0m_c^{\alpha_0},\min_{s\in[s_0,\infty)}\kappa(s)\}.$$
Without loss of generality, assume $c_0m_c^{\alpha_0}\leq\min\{\min_{s\in[s_0,\infty)}\kappa(s),\inf_{\op}\sigma\} $. We obtain from \eqref{rej2} that
\begin{eqnarray}
	\io|\nabla\pet|^2dx+\iop|\nabla\pst|^2dx&\leq&\frac{c}{m_c^{\alpha_0}}\iop\frac{(\theta_\tau(\ct)+\tau)^{\frac{d}{2}}}{\theta_\tau^{\frac{d-1}{2}}(\ct)} G_\tau^2(h,\ct,0)dx\nonumber\\
	&\leq &\frac{c(m_c+\tau)^{\frac{d}{2}}}{ m_c^{\frac{d+1+2\alpha_0}{2}}}\iop H_\tau^2(h,\ct,0)dx\nonumber\\
	&\leq &c\left(\lt\right)^{\frac{d+1+2\alpha_0}{2}}\iop H_\tau^2(h,\ct,0)dx.\label{ese8}
\end{eqnarray}
Note from \eqref{es2} that
\begin{eqnarray}
	\left(\io\pet dx\right)^2+	\left(\iop\pst dx\right)^2&=&\frac{1}{2}\left[	\left(\io(\pet+\pst\chi_{\Omega^\prime}) dx\right)^2+\left(\io(\pet-\pst\chi_{\Omega^\prime}) dx\right)^2\right]\nonumber\\
	&=&\frac{1}{2}\left(\iop(\pet-\pst) dx\right)^2+\frac{1}{2}\left(\int_{\Omega\setminus\op}\pet dx\right)^2.\label{ese5}
\end{eqnarray}
We estimate from the Sobolev inequality that
\begin{eqnarray}
\lefteqn{	\io\left(\left(\pet\right)^2+\left(\pst\right)^2\chi_{\Omega^\prime}\right)dx}\nonumber\\
&\leq&2\io\left(\pet-\frac{1}{|\Omega|}\io\pet dx\right)^2dx+2\iop\left(\pst-\frac{1}{|\op|}\iop\pst dx\right)^2dx\nonumber\\
	&&+\frac{2}{|\Omega|}\left(\io\pet dx\right)^2+\frac{2}{|\op|}\left(\iop\pst dx\right)^2\nonumber\\
	&\leq&c\io|\nabla\pet|^2dx+c\iop|\nabla\pst|^2dx+\frac{2}{|\op|}\left[\left(\io\pet dx\right)^2+	\left(\iop\pst dx\right)^2\right]\nonumber\\
	&\leq&c \left(\lt\right)^{\frac{3d+3+2\alpha_0}{2}}\iop H_\tau^2(h,\ct,0)dx+\frac{1}{|\op|}\left(\int_{\Omega\setminus\op}\pet dx\right)^2.\label{ese6}
\end{eqnarray}
The last step is due to \eqref{ese8}, \eqref{ese4}, and \eqref{ese5}. We easily check
$$\frac{1}{|\op|}\left(\int_{\Omega\setminus\op}\pet dx\right)^2\leq \frac{|\Omega_s|}{|\op|}\io\left(\pet\right)^2dx.$$ Making use of this in \eqref{ese6} and keep (H9) in mind yield the lemmas.\end{proof}
\begin{lemma}For each $q>\frac{N}{2}$ there is a positive number $c$ such that
	\begin{eqnarray}
	\esup_\Omega |\pet|+\esup_{\op} |\pst|
	\leq c\left(\lt\right)^{\frac{qN\alpha_0}{2q-N}+\frac{3d+3+2\alpha_0}{4}}\|H_\tau(h,\ct ,0)\|_{q,\op}.\label{lesb}
	\end{eqnarray}
	\end{lemma}
\begin{proof}Let $q$ be given as in the lemma. Set
$$b=\|H_\tau(h,\ct ,0)\|_{q,\op}.$$
Fix $n>0$. We use $\left(\left(\pet\right)^++b\right)^n$ as a test function in \eqref{lapp1} to get
\begin{eqnarray}\label{uve3}
	\lefteqn{	\frac{4n}{(n+1)^2}\io\kact\left|\nabla\left(\left(\pet\right)^++b\right)^{\frac{n+1}{2}}\right|^2dx+\tau\io \pet\left(\left(\pet\right)^++b\right)^{n}dx}\nonumber\\
	&=&\frac{1}{2} \alpha_4\iop H_\tau(h,\ct,\pst-\pet)\left(\left(\pet\right)^++b\right)^ndx.
\end{eqnarray}
By \eqref{es1},
$$\tau\io \left(\pet\right)^-dx=\tau\io \left(\pet\right)^+dx-\frac{1}{2} \alpha_4\iop H_\tau(h,\ct,\pst-\pet)dx.$$
With this in mind, we calculate
\begin{eqnarray}
	\tau\io \pet\left(\left(\pet\right)^++b\right)^{n}dx&=&	\tau\io \left(\pet\right)^+\left(\left(\pet\right)^++b\right)^{n}dx-b^n	\tau\io \left(\pet\right)^-dx\nonumber\\
	&=&	\tau\io \left(\pet\right)^+\left[\left(\left(\pet\right)^++b\right)^{n}-b^n\right]dx\nonumber\\
	&&+\frac{1}{2} \alpha_4 b^n\iop H_\tau(h,\ct,\pst-\pet)
	dx.\nonumber
\end{eqnarray}
Plug this into \eqref{uve3} to deduce
\begin{eqnarray}
	\lefteqn{	\frac{4n}{(n+1)^2}\io\kact\left|\nabla\left(\left(\pet\right)^++b\right)^{\frac{n+1}{2}}\right|^2dx}\nonumber\\
	&&+\tau\io \left(\pet\right)^+\left[\left(\left(\pet\right)^++b\right)^{n}-b^n\right]dx\nonumber\\
	&=&
	\frac{1}{2} \alpha_4\iop H_\tau(h,\ct,\pst-\pet)\left[\left(\left(\pet\right)^++b\right)^n-b^n\right]dx.\label{uve4}
\end{eqnarray}
Similarly, we can derive from \eqref{psi9} and \eqref{psi10} that
\begin{eqnarray}
	\lefteqn{\frac{4n}{(n+1)^2}\iop\sigma\left|\nabla\left(\left(\pst\right)^++b\right)^{\frac{n+1}{2}}\right|^2dx+\tau\io \left(\pst\right)^+\left[\left(\left(\pst\right)^++b\right)^{n}-b^n\right]dx}\nonumber\\
	&=&-\frac{1}{2} \alpha_4\iop H_\tau(h,\ct,\pst-\pet)\left[\left(\left(\pst\right)^++b\right)^n-b^n\right]dx.\label{uve5}
\end{eqnarray}
Since $G_\tau(y_1,y_2,y_3)$ is increasing in $y_3$, we have
$$N_\tau\equiv \left[G_\tau(h,\ct,\pst-\pet)-G_\tau(h,\ct ,0)\right]\left[\left(\left(\pst\right)^++b\right)^n-\left(\left(\pet\right)^++b\right)^n\right]\geq 0.$$
Add \eqref{uve5} to \eqref{uve4} and keep the above inequality and \eqref{gpr} in mind to deduce
\begin{eqnarray}
	\lefteqn{	\frac{4n}{(n+1)^2}\io\kact\left|\nabla\left(\left(\pet\right)^++b\right)^{\frac{n+1}{2}}\right|^2dx}\nonumber\\
	&&+	\frac{4n}{(n+1)^2}\iop\sigma\left|\nabla\left(\left(\pst\right)^++b\right)^{\frac{n+1}{2}}\right|^2dx\nonumber\\
	&\leq&
	-\frac{1}{2} \alpha_4\iop \sqrt{\theta_\tau(\ct)}N_\tau dx\nonumber\\
	&&+\frac{1}{2} \alpha_4\iop H_\tau(h,\ct ,0)\left[\left(\left(\pst\right)^++b\right)^n-\left(\left(\pet\right)^++b\right)^n\right]dx\nonumber\\
	&\leq&\frac{1}{2} \alpha_4b\left\|\left(\left(\pst\right)^++b\right)^n-\left(\left(\pet\right)^++b\right)^n\right\|_{\frac{q}{q-1},\op}\nonumber\\
	&\leq&\alpha_4\left\|(W+b)^{n+1}\right\|_{\frac{q}{q-1},\Omega},\label{uve6}
\end{eqnarray}
where 
\begin{equation}\label{wdef1}
W_\tau=\max\{\left(\pet\right)^+,\left(\pst\right)^+\chi_{\Omega^\prime}\}.	
\end{equation}
By the Sobolev inequality, we have
\begin{eqnarray}
	\lefteqn{\|(W_\tau+b)^{n+1}\|_{\frac{N}{N-2},\Omega}}\nonumber\\
	&=&\left(\int_{\left\{\left(\pet\right)^+\geq \left(\pst\right)^+\chi_{\Omega^\prime}\right\}}\left(\left(\pet\right)^++b\right)^{\frac{(n+1)N}{N-2}}dx\right)^{\frac{N-2}{N}}\nonumber\\
	&&+\left(\int_{\left\{\left(\pet\right)^+< \left(\pst\right)^+\chi_{\Omega^\prime}\right\}}\left(\left(\pst\right)^+\chi_{\Omega^\prime}+b\right)^{\frac{(n+1)N}{N-2}}dx\right)^{\frac{N-2}{N}}\nonumber\\
	&\leq&	\left(\io \left(\left(\pet\right)^++b\right)^{\frac{n+1}{2}\frac{2N}{N-2}}dx\right)^{\frac{N-2}{N}}+	\left(\iop \left(\left(\pst\right)^++b\right)^{\frac{n+1}{2}\frac{2N}{N-2}}dx\right)^{\frac{N-2}{N}}\nonumber\\
	&\leq& c\io\left|\nabla \left(\left(\pet\right)^++b\right)^{\frac{n+1}{2}}\right|^2dx+c\io \left(\left(\pet\right)^++b\right)^{n+1}dx\nonumber\\
	&&+ c\iop\left|\nabla \left(\left(\pst\right)^++b\right)^{\frac{n+1}{2}}\right|^2dx+c\iop \left(\left(\pst\right)^++b\right)^{n+1}dx\nonumber\\
	&\leq&\left(\frac{c(n+1)^2\alpha_4}{4nm_c^{\alpha_0}}+c|\Omega|^{\frac{1}{q}}\right)\left\|(W_\tau+b)^{n+1}\right\|_{\frac{q}{q-1},\Omega}.\label{uve7}
\end{eqnarray}
The last step is due to \eqref{uve6} and \eqref{wdef1}. Remember
$$\ell\equiv\frac{N}{N-2} \frac{q-1}{q}>1.$$
Then let $n+1=\ell^i, \ i=1,2,\cdots$. Subsequently,
$$\frac{n+1}{n}\leq \frac{\ell}{\ell-1}.$$
Take the $(n+1)^{\textup{th}}$ root of both sides of \eqref{uve7} to get
\begin{eqnarray}
	\|W_\tau+b\|_{\frac{q \ell^{i+1}}{q-1},\Omega}&=&
	\|W_\tau+b\|_{\frac{N(n+1)}{N-2},\Omega}\nonumber\\
	&\leq&\left(\frac{c(n+1)^2\alpha_4}{4nm_c^{\alpha_0}}+c|\Omega|^{\frac{1}{q}}\right)^{\frac{1}{n+1}}\left\|W_\tau+b\right\|_{\frac{q(n+1)}{q-1},\Omega}\nonumber\\
	&\leq &\left(c\left(\lt\right)^{\alpha_0}\right)^{\frac{1}{n+1}}(n+1)^{\frac{1}{n+1}}\left\|W_\tau+b\right\|_{\frac{q(n+1)}{q-1},\Omega}\nonumber\\
	&=&\left(c\left(\lt\right)^{\alpha_0}\right)^{\frac{1}{\ell^i}}\ell^{\frac{i}{\ell^i}}\left\|W_\tau+b\right\|_{\frac{q\ell^i}{q-1},\Omega}.\nonumber
\end{eqnarray}
 Iterate on $i$ to yield
$$
\|W_\tau+b\|_{\frac{q \ell^{i+1}}{q-1},\Omega}\leq \left(c\left(\lt\right)^{\alpha_0}\right)^{\sum_{j=1}^{i}\frac{1}{\ell^j}}\ell^{\sum_{j=1}^{i}\frac{j}{\ell^j}}\left\|W_\tau+b\right\|_{\frac{q\ell}{q-1},\Omega}.$$
Take $i\ra\infty$ to get
\begin{equation}\label{uve8}
	\esup_\Omega (W_\tau+b)\leq c\left(\lt\right)^{\frac{\alpha_0}{\ell-1}}\left\|W_\tau+b\right\|_{\frac{q\ell}{q-1},\Omega}.
\end{equation}
The interpolation inequality in (\cite{GT}, p.146) asserts
$$\left\|W_\tau+b\right\|_{\frac{q\ell}{q-1},\Omega}\leq \vep\esup_\Omega (W_\tau+b)+\frac{1}{\vep^{\frac{q\ell-q+1}{q-1}}}\left\|W_\tau+b\right\|_{1,\Omega},\ \ \vep>0.$$
Use this in \eqref{uve8} and choose $\vep$ appropriately in the resulting inequality to get
\begin{eqnarray}
	\esup_\Omega W_\tau&\leq&\esup_\Omega (W_\tau+b)\nonumber\\
	&\leq&c\left(\lt\right)^{\frac{q\ell\alpha_0}{(\ell-1)(q-1)}}\left\|W_\tau+b\right\|_{1,\Omega}\nonumber\\
	&\leq&c\left(\lt\right)^{\frac{qN\alpha_0}{2q-N}}\left(\left\|W_\tau\right\|_{1,\Omega}+\|H_\tau(h,\ct ,0)\|_{q,\op}\right).\label{uve9}
\end{eqnarray}
According to  \eqref{esl2}, we have
\begin{eqnarray}
	\|W_\tau\|_{\Omega,1}&\leq& 	\|\left(\pet\right)^+\|_{\Omega,1}+	\|\left(\pst\right)^+\|_{\op,1}\nonumber\\
&\leq&c\left(\lt\right)^{\frac{3d+3+2\alpha_0}{4}}\|H_\tau(h,\ct,0)\|_{2,\op}.\nonumber
\end{eqnarray}
This together with \eqref{uve9} implies 
\begin{eqnarray}
	\esup_\Omega \left(\pet\right)^++\esup_{\op} \left(\pst\right)^+
\leq c\left(\lt\right)^{\frac{qN\alpha_0}{2q-N}+\frac{3d+3+2\alpha_0}{4}}\|H_\tau(h,\ct ,0)\|_{q,\op}.\label{uve111}
\end{eqnarray}
In an entirely similar manner, we can also prove
\begin{eqnarray}
		\esup_\Omega \left(\pet\right)^-+\esup_{\op} \left(\pst\right)^-
	\leq c\left(\lt\right)^{\frac{qN\alpha_0}{2q-N}+\frac{3d+3+2\alpha_0}{4}}\|H_\tau(h,\ct ,0)\|_{q,\op}.\nonumber
\end{eqnarray}
Combining this with\eqref{uve111} yields \eqref{lesb}. The proof is complete.
\end{proof}

To continue the proof of Theorem \ref{thm2}, we let
$$\mt=\mt(T)=\max\{\max_{\overline{\ot}}\ct,1\}\geq 1.$$
We can easily see from \eqref{htdf} that
$$\left|H_\tau(h,\ct,0)\right|\leq c\left(\left(\mt\right)^{\frac{1}{2}}+\left(\mt\right)^{d+\frac{1}{2}}\right)\left(\lt\right)^d\leq c\left(\mt\right)^{d+\frac{1}{2}}\left(\lt\right)^d.$$
This together with \eqref{lesb} implies
$$	\esup_\Omega |\pet|+\esup_{\op} |\pst|
\leq c\left(\mt\right)^{d+\frac{1}{2}}\left(\lt\right)^{\frac{qN\alpha_0}{2q-N}+\frac{7d+3+2\alpha_0}{4}}.$$
We conclude from Lemmas \ref{lcub} and \ref{lclb} that
\begin{eqnarray}
\mt&\leq&c(1+T)^{\frac{2N}{(N+2)}}T^{\frac{2(2q-2-N)}{q(N+2)}}\left\|e^{\alpha_2(\ps-\pe)}\right\|_{q,\opt}^{2}+4\|C_0\|_{\infty,\Omega}+2\nonumber\\
&\leq&c(1+T)^{\frac{2N}{(N+2)}}T^{\frac{4}{N+2}}e^{c\left(\mt\right)^{d+\frac{1}{2}}\left(\lt\right)^{\frac{qN\alpha_0}{2q-N}+\frac{7d+3+2\alpha_0}{4}}}+4\|C_0\|_{\infty,\Omega}+2,\label{dff1}\\
\lt	&\leq& c(1+T)^{\frac{2N}{(N+2)(2d-1)}}\left\|e^{-\alpha_2(\pst-\pet)}\right\|_{q,\opt}^{\frac{2}{2d-1}}+\frac{1}{\min_{\overline{\Omega}}C_0}\nonumber\\
&\leq & c(1+T)^{\frac{2N}{(N+2)(2d-1)}}T^{\frac{2}{q(2d-1)}}e^{c\left(\mt\right)^{d+\frac{1}{2}}\left(\lt\right)^{\frac{qN\alpha_0}{2q-N}+\frac{7d+3+2\alpha_0}{4}}}+\frac{1}{\min_{\overline{\Omega}}C_0}.\label{dff2}
\end{eqnarray}
Set
\begin{eqnarray}
	\gamma&=&\frac{qN\alpha_0}{2q-N}+\frac{7d+3+2\alpha_0}{4}>d+\frac{1}{2}>1,\label{r16}\\
	g(T)&=&\max\left\{(1+T)^{\frac{2N}{(N+2)}}T^{\frac{2(2q-2-N)}{q(N+2)}}, (1+T)^{\frac{2N}{(N+2)(2d-1)}}T^{\frac{2}{q(2d-1)}}\right\}.\label{r17}
\end{eqnarray}
We can conclude from \eqref{dff1} and \eqref{dff2} that there is a positive number $c$ such that
\begin{eqnarray}
	\mt\lt< cg^2(T)e^{c\left(\mt\lt\right)^\gamma}+c.\nonumber
\end{eqnarray}
We are in a position to invoke Lemma \ref{prop2.2}. For this purpose set
$$a(s)=\mt(s)\lt(s),\ \ \ep=cg^2(T).$$
Subsequently,
$$a(s)< cg^2(s)e^{ca^\gamma(s)}+c\leq \ep e^{ca^\gamma(s)}+c\ \ \mbox{for each $s\in [0,T]$.}$$
Here we have used the fact that $g(T)$ strictly increases from $0$ to infinity on the interval $[0,\infty)$. We are in a position to apply Lemma \ref{prop2.2}. Let $\ep_0, s_0$ be determined by Lemma \ref{prop2.2}. Obviously, there is a unique $\tmax\in(0,\infty)$ such that
\begin{equation}\label{tzdf}
	\ep_0=cg^2(\tmax).
\end{equation}
This implies
$$a(s)< s_0\ \ \mbox{for each $s\in [0,\tmax]$ since we obviously have $a(0)\leq c$.}$$
Consequently,
\begin{equation}\label{tzdf1}
	\frac{1}{c}\leq \ct\leq c\ \ \mbox{in $\Omega_{\tmax}$ for some $c>0$.}
\end{equation}
\begin{lemma}Let $\tmax$ be given as in \eqref{tzdf}. For each $p>1$ the sequence $\{\pet\}$ is precompact in $L^p(\Omega_{\tmax})$ and the sequence $\{\pst\}$ is precompact in $L^p(\Omega^\prime_{\tmax})$.
\end{lemma}
\begin{proof} Let $t_1, t_2\in [0,\tmax]$. For simplicity, we employ the following notions
	\begin{equation}
	h_t(x)=h(x,t),\	\ct_{t}(x)=\ct(x,t),\ \left(\pet\right)_{t}(x)=\pet(x,t),\ \left(\pst\right)_t(x)=\pst(x,t).\nonumber
	\end{equation}
We can derive from 	\eqref{lapp1} and \eqref{lapp2} that
	\begin{eqnarray}
		\lefteqn{	-\mdiv\left[\kappa_\tau(\ct_{t_1})\nabla\left[\left(\pet\right)_{t_1}-\left(\pet\right)_{t_2}\right]\right]+\tau\left(\left(\pet\right)_{t_1}-\left(\pet\right)_{t_2}\right)}\nonumber\\
		&=&\mdiv\left[\left(\kappa_\tau(\ct_{t_1})-\kappa_\tau(\ct_{t_2})\right)\nabla\left(\pet\right)_{t_2}\right]+\frac{1}{2}\alpha_4H_\tau\left((h_{t_1},\ct_{t_1},\left(\pst\right)_{t_1}-\left(\pet\right)_{t_1}\right)\chi_{\Omega^\prime}\nonumber\\
		&&-\frac{1}{2}\alpha_4H_\tau\left(h_{t_2},\ct_{t_2},\left(\pst\right)_{t_2}-\left(\pet\right)_{t_2}\right)\chi_{\Omega^\prime} \ \mbox{in $\Omega$},\label{def10}\\
		\lefteqn{	-\mdiv\left[\sigma\nabla\left[\left(\pst\right)_{t_1}-\left(\pst\right)_{t_2}\right]\right]+\tau\left(\left(\pst\right)_{t_1}-\left(\pst\right)_{t_2}\right)}\nonumber\\
		&=&-\frac{1}{2}\alpha_4H_\tau\left((h_{t_1},\ct_{t_1},\left(\pst\right)_{t_1}-\left(\pet\right)_{t_1}\right)\nonumber\\
		&&+\frac{1}{2}\alpha_4H_\tau\left(h_{t_2},\ct_{t_2},\left(\pst\right)_{t_2}-\left(\pet\right)_{t_2}\right) \ \mbox{in $\op$}.\label{def11}
	\end{eqnarray}
	Observe that
	\begin{eqnarray}
		\lefteqn{H_\tau\left((h_{t_1},\ct_{t_1},\left(\pst\right)_{t_1}-\left(\pet\right)_{t_1}\right)-H_\tau\left(h_{t_2},\ct_{t_2},\left(\pst\right)_{t_2}-\left(\pet\right)_{t_2}\right)}\nonumber\\
		&=&H_\tau\left((h_{t_1},\ct_{t_1},\left(\pst\right)_{t_1}-\left(\pet\right)_{t_1}\right)-H_\tau\left(h_{t_2},\ct_{t_1},\left(\pst\right)_{t_1}-\left(\pet\right)_{t_1}\right)\nonumber\\
		&&+H_\tau\left(h_{t_2},\ct_{t_1},\left(\pst\right)_{t_1}-\left(\pet\right)_{t_1}\right)-H_\tau\left(h_{t_2},\ct_{t_2},\left(\pst\right)_{t_1}-\left(\pet\right)_{t_1}\right)\nonumber\\
		&&+H_\tau\left(h_{t_2},\ct_{t_2},\left(\pst\right)_{t_1}-\left(\pet\right)_{t_1}\right)-H_\tau\left(h_{t_2},\ct_{t_2},\left(\pst\right)_{t_2}-\left(\pet\right)_{t_2}\right)\nonumber\\
		&\equiv&J_1+J_2+J_3.\nonumber
	\end{eqnarray}
We can easily conclude from \eqref{tzdf1} and \eqref{lesb} that
\begin{eqnarray}
	|J_1|&\leq &c\|h_{t_1}-h_{t_2}\|_{\infty,\op},\nonumber\\
	|J_2|&\leq &c\|\ct_{t_1}-\ct_{t_2}\|_{\infty,\op}.\nonumber
\end{eqnarray}
By an argument similar to \eqref{rej4}, we can obtain
\begin{eqnarray}
	\lefteqn{J_3\left[\left[\left(\pst\right)_{t_1}-\left(\pet\right)_{t_1}\right]-\left[\left(\pst\right)_{t_2}-\left(\pet\right)_{t_2}\right]\right]}\nonumber\\
	&\geq &c\left[\left[\left(\pst\right)_{t_1}-\left(\pst\right)_{t_2}\right]-\left[\left(\pet\right)_{t_1}-\left(\pet\right)_{t_2}\right]\right]^2.\nonumber
\end{eqnarray}
Use $\left(\pet\right)_{t_1}-\left(\pet\right)_{t_2}$ as a test function in \eqref{def10}, $\left(\pst\right)_{t_1}-\left(\pst\right)_{t_2}$ in \eqref{def11},  add up the two resulting equations to deduce
\begin{eqnarray}
	\lefteqn{\io\left|\nabla\left[\left(\pet\right)_{t_1}-\left(\pet\right)_{t_2}\right]\right|^2dx+\iop\left|\nabla\left[\left(\pst\right)_{t_1}-\left(\pst\right)_{t_2}\right]\right|^2dx}\nonumber\\
	&&c\iop\left[\left[\left(\pst\right)_{t_1}-\left(\pst\right)_{t_2}\right]-\left[\left(\pet\right)_{t_1}-\left(\pet\right)_{t_2}\right]\right]^2dx\nonumber\\
	&\leq&c\|\kappa(\ct_{t_1}+\tau)-\kappa(\ct_{t_2}+\tau)\|_{\infty,\op}+c\|h_{t_1}-h_{t_2}\|_{\infty,\op}+c\|\ct_{t_1}-\ct_{t_2}\|_{\infty,\op}.\nonumber
\end{eqnarray}
Integrate \eqref{def10} over $\Omega$ and \eqref{def11} over $\op$ and add up the two resulting equations to get
$$\io \left[\left(\pet\right)_{t_1}-\left(\pet\right)_{t_2}\right]dx+\iop\left[\left(\pst\right)_{t_1}-\left(\pst\right)_{t_2}\right]dx=0.$$
This puts us in a position to employ the proof of \eqref{esl2}. Upon doing so, we arrive at
\begin{eqnarray}
	\lefteqn{\io\left[\left(\pet\right)_{t_1}-\left(\pet\right)_{t_2}\right]^2dx+\iop\left[\left(\pst\right)_{t_1}-\left(\pst\right)_{t_2}\right]^2dx}\nonumber\\
	&\leq&c\|\kappa(\ct_{t_1}+\tau)-\kappa(\ct_{t_2}+\tau)\|_{\infty,\op}+c\|h_{t_1}-h_{t_2}\|_{\infty,\op}+c\|\ct_{t_1}-\ct_{t_2}\|_{\infty,\op}.\label{def5}
\end{eqnarray}
Classical regularity results for linear parabolic equations (\cite{LSU}, Chap. III) assert that $\{\ct\}$ is bounded in $C^{\frac{\delta}{2},\delta}(\overline{\Omega_{\tmax}})$ for some $\delta\in(0,1)$. This combined with \eqref{def5} implies that 
$\{\pet\}$ is uniformly equicontinuous in $C([0,\tmax]; L^2(\Omega))$ and $\{\pst\}$ in $C([0,\tmax]; L^2(\op))$. Furthermore, $\{\pet\}$ is bounded in $L^\infty(0, \tmax; W^{1,2}(\Omega))$ and $\{\pet\}$ in $L^\infty(0, \tmax; W^{1,2}(\op))$. By virtue of the Ascoli theorem (\cite{S}, Lemma 1), the two sequences are precompact in $C([0,\tmax]; L^2(\Omega))$ and $C([0,\tmax]; L^2(\op))$, respectively. This implies the lemma.
	\end{proof}
Now we have all the ingredients necessary to pass to the limit in \eqref{app1}-\eqref{app6}. The proof of Theorem \ref{thm1} is completed.

\end{document}